\documentclass[12pt,a4paper]{amsart}
\usepackage{amsfonts}
\usepackage{amssymb}
\usepackage{amsmath}
\usepackage{amsthm}
\usepackage{cite}
\usepackage{graphicx}
\usepackage{stmaryrd}
\usepackage{amscd}
\usepackage{color}
\usepackage{mathtools} 
\usepackage{cancel}
\usepackage[normalem]{ulem}
\usepackage{cancel}
\usepackage{hyperref}
\newtheorem{theorem}{Theorem}
\theoremstyle{plain}

\newtheorem{definition}[theorem]{Definition}

\newtheorem{corollary}[theorem]{Corollary}
\newtheorem{lemma}[theorem]{Lemma}
\newtheorem{remark}[theorem]{Remark}

\numberwithin{equation}{section}

\newcommand{\R}{\mathbb{R}}
\newcommand{\C}{\mathbb{C}}

\newcommand{\mS}{\mathbb{S}}
\newcommand{\N}{\mathbb{N}}

\newcommand{\Hcurl}{H_{\rm curl}}

\renewcommand{\phi}{\varphi}

\newcommand{\bpm}{{\begin{pmatrix}}}
\renewcommand{\phi}{\varphi}

\DeclareMathOperator{\const}{const.}

\DeclareMathOperator{\range}{Rg}

\DeclareMathOperator{\trace}{trace}

\newcommand{\dist}{\text{\rm dist}}
\newcommand{\supp}{\text{\rm supp}}

\newcommand{\Rg}{\text{\rm Rg}}

\def\blem{\begin{lemma}}\def\elem{\end{lemma}}
\def\bthm{\begin{theorem}}\def\ethm{\end{theorem}}
\def\bcor{\begin{corollary}}\def\ecor{\end{corollary}}
\def\beq{\begin{equation}}\def\eeq{\end{equation}}

\newtagform{pm}{(}{)\rlap{$_\pm$}}

\begin{document}

\title[Breathers and rogue waves for semilinear curl-curl wave equations]{Breathers and rogue waves for semilinear curl-curl wave equations}

\author{Michael Plum}
\address{M. Plum \hfill\break 
Institute for Analysis, Karlsruhe Institute of Technology (KIT), \hfill\break
D-76128 Karlsruhe, Germany}
\email{michael.plum@kit.edu}

\author{Wolfgang Reichel}
\address{W. Reichel \hfill\break 
Institute for Analysis, Karlsruhe Institute of Technology (KIT), \hfill\break
D-76128 Karlsruhe, Germany}
\email{wolfgang.reichel@kit.edu}

\date{\today}


\begin{abstract} We consider localized solutions of variants of the semilinear curl-curl wave equation 
$s(x) \partial_t^2 U  +\nabla\times\nabla\times U + q(x) U \pm V(x) |U|^{p-1} U = 0$ for $(x,t)\in \R^3\times\R$ and arbitrary $p>1$. Depending on the coefficients $s, q, V$ we can prove the existence of three types of localized solutions: time-periodic solutions decaying to $0$ at spatial infinity, time-periodic solutions tending to a nontrivial profile at spatial infinity (both types are called breathers), and rogue waves which converge to $0$ both at spatial and temporal infinity. Our solutions are weak solutions and take the form of gradient fields. Thus they belong to the kernel of the curl-operator so that due to the structural assumptions on the coefficients the semilinear wave equation is reduced to an ODE. Since the space dependence in the ODE is just a parametric dependence we can analyze the ODE by phase plane techniques and thus establish the existence of the localized waves described above. Noteworthy side effects of our analysis are the existence of compact support breathers and the fact that one localized wave solution $U(x,t)$ already generates a full continuum of phase-shifted solutions $U(x,t+b(x))$ where the continuous function $b:\R^3\to\R$ belongs to a suitable admissible family.
\end{abstract}
\maketitle


\section{Introduction}
Localized solutions of nonlinear wave equations on $\R^d\times\R$ have been a research topic for many decades both in the physics and the mathematics community. One type of such solutions are breathers which are time-periodic solutions with localization in space. Another type of solutions are rogue waves which are localized both in space and time. Both types have been known for the nonlinear Schr\"odinger equation (NLS) for a long time, i.e., bright breathers in the focusing case \cite{zakharov_shabat_71}, dark/black breathers in the defocusing case \cite{zakharov_shabat_73}, and the Peregrine solution \cite{peregrine} as an example for a rogue wave also in the focusing case. For other completely integrable equations like the Korteweg-de-Vries equation and the sine-Gordon equation bright breather solutions are known, cf. \cite{ab_kaup_newell_segur:73,kruskal_kdv}. Whereas (at least in the scalar case) both bright and dark/black breather solutions converge at spatial infinity to a (possibly zero) limit, the terminology separates them with regards to the profile of their absolute value of their intensity: for bright breathers the absolute value stands above the limit at infinity, whereas for the dark breathers the absolute value of the intensity appears as a relative dip below the background limit at infinity. 

\medskip

For the above mentioned completely integrable systems the inverse scattering transform developed in \cite{kruskal_kdv} is the key to finding explicit localized solutions. For wave equations where the nonlinearity is $\mS^1$-equivariant it is much simpler to find  complex valued time-harmonic breathers of the type $u(x,t) = e^{i\omega t} \pmb{u}(x)$. They are much easier to obtain by variational methods and have been studied extensively for the nonlinear Schr\"odinger equation, cf. \cite{berestycki_lions, strauss:1977}, and more recently for spatially periodic potentials in e.g. \cite{alama_li:1992,pankov:2005}. In our case it is also very simple to find monochromatic, time-harmonic, complex valued breathers, cf. Remark~\ref{remark_main_breather}(c) and Remark~\ref{remark_main_breather_dark} below, and it is remarkable that this is true under (almost) exactly the same assumptions as for the existence of polychromatic, real valued breathers in Theorem~\ref{main_breather} and in Theorem~\ref{main_breather_dark}. 

\medskip

As soon as one steps away from completely integrable systems or from time-harmonic breathers, the number of examples of known localized solutions of nonlinear wave equations becomes quite small. In discrete nonlinear lattice equations breathers are still more common, cf. \cite{mackay_aubry,james_breathers:09}. For semilinear scalar $1+1$-dimensional wave equations with non-constant coefficients it is a challenging task to find bright breathers. This was accomplished for the first time in \cite{blank_bruckner_lescarret_schneider:11} by making use of spatial dynamics and center manifold reduction, and subsequently by variational methods in \cite{hirsch_reichel_poly, mandel_scheider_breather, scheider_breather, maier_reichel_schneider}. While the physics literature on rogue waves is quite abundant, cf. \cite{efim_pelinovsky} for a text book on oceanic rogues waves, mathematically rigorous proofs for the existence of rogue waves are much more rare. Beyond Peregrine's first example \cite{peregrine} they have been found for integrable systems like derivative NLS, focusing NLS and modified Korteweg-de-Vries in \cite{chen_pel2,chen_pel1, chen_pel_white1}, and recently also for nonintegrable variants of the NLS as perturbations of the Peregrine solution in \cite{Kevrekidis3, Kevrekidis2, Kevrekidis1}.
 
\medskip

In this paper we consider for $p>1$ two $3+1$-dimensional semilinear curl-curl wave equations of the form 
\begin{equation}
\usetagform{pm}
\label{breather}
s(x) \partial_t^2 U  +\nabla\times\nabla\times U + q(x) U \pm V(x) |U|^{p-1} U = 0 \mbox{ for } (x,t)\in \R^3\times\R
\end{equation}
and 
\begin{equation}
\label{rogue}
s(x) \partial_t^2 U  +\nabla\times\nabla\times U - q(x) U + V(x) |U|^{p-1} U = 0 \mbox{ for } (x,t)\in \R^3\times\R.
\end{equation}
Here we assume that $s,q,V:\R^3\to (0,\infty)$ are positive functions. Note that \eqref{breather}$_\pm$ and \eqref{rogue} differ in the sign in front of the linear term $q(x) U$ and that in \eqref{breather}$_\pm$ the sign in front of the nonlinearity may be positive or negative whereas in \eqref{rogue} only positive coefficients in front of the nonlinearity are admissible. All our results apply verbatim when the curl-curl operator in \eqref{breather}, \eqref{rogue} is replaced by a differential expression $P(x,\nabla)\nabla\times$ where $P(x,y)=(P_{ij}(x,y))_{i,j=1,2,3}$ is a $3\times 3$-matrix with $P_{ij}(x,y))$, $i,j=1,\ldots,3$ being polynomials in $y=(y_1,y_2,y_3)$ with $x$-dependent coefficients. Examples are $P(x)=\nabla\times A(x) \nabla\times$, $P(x)=A(x)\nabla\times\nabla\times$ with a $3\times 3$-matrix $A(x)$.

\medskip

Both \eqref{breather}$_\pm$ and \eqref{rogue} have certain similarities with the second-order form of the cubic nonlinear Maxwell problem, and this motivates us to study these equations. In the absence of charges and currents, and for a simplified Kerr-type material law $D=\epsilon_0(1+\chi_1(x)+\chi_3(x) |E|^2)E$, a second-order wave-type equation derived form the nonlinear Maxwell problem for the electric field $E=E(x,t)$ is given by $\nabla\times\nabla\times E+\epsilon_0\mu_0(1+\chi_1(x))\partial_t^2 E+\epsilon_0\mu_0\chi_3(x) (|E|^2 E)_{tt}=0$ in $\R^3\times\R$. Here $\epsilon_0>0$ is the vacuum electric permittivity, $\mu_0>0$ the vacuum magnetic permeability, and $\chi_1, \chi_3:\R\to\R$ denote the first and third order electric response tensors of the material (second order responses do not appear in centrosymmetric materials). The nonlinear Maxwell problem differs from our semilinear equations by the second time derivative of the nonlinear term. However, since the particularity of the appearance of the curl-curl operator instead of the Laplacian is the same for both problems, we are motivated to investigate the effects of the curl-curl operator. In fact, most of our results are only true because the appearance of the curl operator.

\medskip

For \eqref{breather}$_\pm$, where both the $+$ and the $-$ sign in front of the nonlinear term $V(x)|U|^{p-1}U$ is admissible, as well as for \eqref{rogue} we will establish the existence of real-valued breathers (time-periodic spatially localized solutions). The difference between the breathers for the equations \eqref{breather}$_\pm$ and \eqref{rogue} is this: for \eqref{breather}$_\pm$ we will show existence of breathers which decay to $0$ at spatial infinity whereas for \eqref{rogue} we will find breathers which have a nonzero profile at spatial infinity. One might be inclined to call the first type of breathers (decaying to $0$) ``bright breathers'' and the second type (having a nonzero profile at infinity) ``dark breathers'' but we will not do this, since in general we have no evidence that the absolute value of the intensity of the second type has a dip which is below the absolute value of the intensity at spatial infinity.  In contrast to the existence results for breathers, the existence of rogue waves (solutions which are localized in space and time) will only be done for \eqref{rogue}. 

Our results depend strongly on the assumption that (at least some of) the coefficients $s(x), q(x), V(x)$ in \eqref{breather}$_\pm$ and \eqref{rogue} depend on $x$. In this sense we are close to the results in \cite{blank_bruckner_lescarret_schneider:11} and \cite{hirsch_reichel_poly,maier_reichel_schneider} since these papers also make strong use of spatially varying coefficients. However, even more important in our setting is the particular property of the curl-operator to annihilate gradient fields. This enables us to construct gradient field breathers by ODE-techniques which lie in the kernel of the curl-operator.

\medskip

Results for real valued localized solutions for \eqref{breather}$_\pm$ or \eqref{rogue} are rare. To the best of our knowledge we are only aware our our previous paper \cite{plum_reichel}. If instead of real valued breathers one considers complex valued breathers of the time harmonic type $U(x,t) = e^{i\omega t} \pmb{U}(x)$ then the resulting elliptic equation for the profile $\pmb{U}$ is the same whether one starts from \eqref{breather}$_\pm$, \eqref{rogue} or from the above mentioned nonlinear Maxwell problem. For the latter there are numerous results on the existence of time harmonic solutions relying on refined methods for vector valued elliptic variational problems, cf. \cite{ABDF06,benci_fortunato_archive,daprile_siciliano,mederski_schino,mederski_schino_szulkin,mederski_overview, mederski_arma,mederski_reichel}, or fixed point methods, cf. \cite{mandel_curl_curl}. 

\medskip

In our previous paper \cite{plum_reichel} we have considered the situation where $s, q, V$ are radially symmetric coefficients. Accordingly, we constructed classical solutions to \eqref{breather}$_\pm$ of the form
$$
U(x,t) = \psi(|x|,t)\frac{x}{|x|}.
$$
which has the property that $\nabla\times\nabla\times U=0$ since $U$ is a gradient field. In the current paper we generalize this approach by making generalized symmetry assumptions given below on the coefficients $s, q, V$ and using the ansatz 
$$
U(x,t) = \psi(g(x),t)\frac{\nabla g(x)}{|\nabla g(x)|}
$$
with the additional assumption that $|\nabla g(x)|=G(g(x))$ and with suitable regularity conditions on $g:\R^3\to \R$ and $G:\R\to (0,\infty)$. Again $U$ is a gradient field and thus $\nabla\times\nabla\times U=0$. Obviously the choice $g(x)=|x|$ reduces to the previously considered situation. In order to stay methodically close to the previous situation we require a kind of compatibility condition between the function $g$ and the coefficients $s, q, V$. We express the precise compatibility conditions in the following way denoting $\Rg(g)=g(\R^3)$:
\begin{itemize}
\item[(C1)] $g\in W^{1,1}_{loc}(\R^3)$ and is continuous with $\nabla g \not = 0$ a.e. on $\R^3$.
\item[(C2)] $|\nabla g(x)| = G(g(x))$ with $G:\Rg(g)\to (0,\infty)$ continuous, $\inf_{\Rg(g)}G>0$.
\item[(C3)] For $s,q, V:\R^3\to (0,\infty)$ assume $s(x)=\tilde s(g(x))$, $q(x)=\tilde q(g(x))$, $V(x)=\tilde V(g(x))$ with $\tilde s, \tilde q, \tilde V:\Rg(g)\to (0,\infty)$ continuous.
\end{itemize}
We use the abbreviation 
$$
\sigma(x) = \sqrt{\frac{q(x)}{s(x)}} \quad\mbox{ and }\quad \tau(x) = \Bigl(\frac{q(x)}{V(x)}\Bigr)^{\frac{1}{p-1}}
$$
and 
$$
\tilde\sigma(\zeta) = \sqrt{\frac{\tilde q(\zeta)}{\tilde s(\zeta)}} \quad\mbox{ and }\quad \tilde\tau(\zeta) = \Bigl(\frac{\tilde q(\zeta)}{\tilde V(\zeta)}\Bigr)^{\frac{1}{p-1}}
$$
where the functions $\sigma, \tau: \R^3\to (0,\infty)$ and $\tilde\sigma, \tilde \tau:\Rg(g)\to (0,\infty)$ are continuous.

\medskip

Using arbitrary constants $\gamma>0, r_0\geq 0$ and the notation $r=\sqrt{x_1^2+x_2^2}$ we present in Table~\ref{list} examples of functions $g$, $G$ satisfying (C1)-(C2). We also indicate the applicability of Theorem~\ref{main_breather}, Theorem~\ref{main_breather_dark}, and Theorem~\ref{main_rogue} for reasons that will be explained in the remarks following the statement of the theorems. 

\begin{table}[ht!] \label{list}
\begin{tabular}{|c|c|c|c|c|}
\hline
$g(x)$ & $G(x)$ & Theorem~\ref{main_breather} applies & Theorem~\ref{main_breather_dark} applies & Theorem~\ref{main_rogue} applies \\ \hline 
$\gamma |r-r_0|+x_3$ & $\sqrt{1+\gamma^2}$ & no & yes & no \\
$\gamma |r-r_0|+|x_3|$ & $\sqrt{1+\gamma^2}$ & yes & yes & yes \\
$\sqrt{(r-r_0)^2+ x_3^2}$ & $1$ & yes & yes & yes\\ \hline  
\end{tabular}
\caption{Examples of functions $g, G$ fulfilling (C1)-(C2). Applicability of the main theorems is indicated. Notation: $\gamma>0$, $r_0\geq 0$ are arbitrary constants, $r=\sqrt{x_1^2+x_2^2}$.}
\end{table}

In contrast to our previous paper \cite{plum_reichel}, where we considered only classical solutions, the solutions we construct in the present paper are suitably defined weak solutions in a sense made precise below.

\medskip

In Section~\ref{sec:breather} we look for real-valued weak solutions $U:\R^3\times \R\to \R^3$ which are $T$-periodic in time and spatially exponentially localized. The notion of exponential localization in space and in space-time (needed later in Section~\ref{sec:rogue}) is described next.

\begin{definition}[exponential localization] Let $U_\infty:\R^3\times\R\to\R$ be a vector field. A function $U:\R^3\times\R\to\R^3$ is called 
\begin{itemize}
\item[(i)] spatially exponentially localized w.r.t. $U_\infty$ if there exists $\delta>0$ such that
$$\sup\{\bigl|U(x,t)-U_\infty(x,t)\bigr| e^{\delta |x|}: (x,t)\in \R^3\times\R\}< \infty,
$$ 
\item[(ii)] space-time exponentially localized w.r.t. $U_\infty$ if there exists $\delta>0$ such that
$$
\sup\{\bigl|U(x,t)-U_\infty(x,t)\bigr|e^{\delta(|t|+|x|)}: (x,t)\in \R^3\times\R\}<\infty.
$$ 
\end{itemize}
\end{definition}

\begin{remark} Both in Theorem~\ref{main_breather} and Theorem~\ref{main_rogue} the vector field $U_\infty$ will be zero. In Theorem~\ref{main_breather_dark}, where we look for time periodic solutions $U$, the vector field $U_\infty$ will also have to be time periodic. In fact, it will be a nontrivial gradient field which solves \eqref{rogue} at spatial infinity.
\end{remark}

Weak solutions are defined as follows. 

\begin{definition}[weak solution] For $-\infty\leq T_1<T_2\leq +\infty$ a function $U: (T_1,T_2)\to L^\infty_{loc}(\R^3)^3$, which satisfies $U\in L^2((T_1,T_2); \Hcurl(B))\cap H^2((T_1,T_2);L^2(B)^3)$ for every open ball $B\subset\R^3$, is called a weak solution of \eqref{breather}$_\pm$ on $(T_1,T_2)$ if and only if for a.a. $t\in (T_1,T_2)$
$$
\int_{\R^3} \Bigl(s(x)\partial_t^2 U(x,t)\cdot v+ \nabla \times U(x,t)\cdot \nabla\times v+ \bigl(q(x) U(x,t) \pm  V(x) |U(x,t)|^{p-1}U(x,t)\bigr)\cdot v\Bigr)\,dx=0
$$
for all $v\in \Hcurl(\R^3)$ with compact support. Weak solutions of \eqref{rogue} are defined analogously.  
\label{def_breather}
\end{definition}

\begin{remark} For $-\infty<T_1<T_2<+\infty$ a weak solution $U$ is $T=T_2-T_1$-periodic if $U(\cdot,T_1)=U(\cdot,T_2)$, $\partial_t U(\cdot,T_1)=\partial_t U(\cdot,T_2)$, where the equalities are understood in the sense of the $L^2(B)$-equalities $\trace U|_{t=T_1} = \trace U|_{t=T_2}$ and $\trace \partial_t U|_{t=T_1}=\trace \partial_t U|_{t=T_2}$ for every open ball $B\subset\R^3$ since $U,\partial_t U\in H^1((T_1,T_2); L^2(B)^3)$. 
\end{remark}

For the case of \eqref{breather}$_\pm$ we consider both the case of the coefficient $+V(x)$ and $-V(x)$ in front of the nonlinearity. In both cases we have existence results which differ in only one hypothesis. Our main result for breathers reads as follows. 

\begin{theorem} Let $s,q,V:\R^3\to (0,\infty)$ and $g:\R^3\to\R$ satisfy {\rm (C1)--(C3)} and suppose additionally that 
\begin{itemize}
\item[(B1)] $\sigma_\infty:= \lim_{|x|\to\infty}\sigma(x)$ exists and $\sup_{\R^3} |\sigma(x)-\sigma_\infty|e^{\delta|x|}<\infty$ for some $\delta>0$,
\item[(B2)] $\sigma\leq\sigma_\infty$, $\sigma\not\equiv \sigma_\infty$ on $\R^3$,
\item[(B3)] $\sup_{\R^3} \tau<\infty$.
\end{itemize}
Then with $\omega:=\sigma_\infty$ there exists a $T=\frac{2\pi}{\omega}$-periodic $\R^3$-valued weak solution $U\not\equiv 0$ of \eqref{breather}$_+$ (called breather) which is spatially exponentially localized w.r.t. $U_\infty=0$. The breather generates a continuum of phase-shifted breathers $U_a(x,t)=U(x,t+a(g(x)))$ where $a:\R\to \R$ is an arbitrary continuous function.
The same statements hold for \eqref{breather}$_-$ if {\rm (B2)} is replaced by 
\begin{itemize}
\item[(B2)'] $\sigma\geq\sigma_\infty$, $\sigma\not\equiv \sigma_\infty$ on $\R^3$.
\end{itemize}
\label{main_breather}
\end{theorem}

\begin{remark} \label{remark_main_breather}
(a) Let us denote by $D$ the set of all finite accumulation points of $(g(x^{(n)}))_{n\in\N}$ when $(x^{(n)})_{n\in\N}$ is an arbitrary sequence in $\R^3$ with $|x^{(n)}|\to\infty$ as $n\to \infty$. Then assumption (B1) implies that $\tilde\sigma\equiv \omega=\sigma_\infty$ on $D$. For the first example $g(x)=\gamma |r-r_0|+x_3$ in Table~\ref{list} we find $D=\R$ so that $\tilde\sigma\equiv \omega=\sigma_\infty$ and thus (as will become clear from the proof) $U\equiv 0$ is the only outcome of the construction of Theorem~\ref{main_breather}. For the other two examples in Table~\ref{list} one finds $D=\emptyset$. \\
(b) An interesting aspect of Theorem~\ref{main_breather} may be that one can construct breather solutions with compact spatial support when the function $g$ has the property that $g(x)\to\infty$ as $|x|\to\infty$. We will prove this statement after the end of the proof of Theorem~\ref{main_breather} in Remark~\ref{details}. If in addition to the assumptions of Theorem~\ref{main_breather} we suppose that $\tilde\sigma$ is identically equal to $\omega=\frac{2\pi}{T}$ outside the ball $B_R(0)\subset \R^3$, then we will see that the entire family $U_a$ of breather solutions has compact spatial support within $\overline{B}_\rho(0)$ provided $\rho$ is so large that $g(x)\geq R$ for $|x|\geq \rho$.\\
(c) Under the same assumptions as in Theorem~\ref{main_breather} one can directly construct monochromatic complex-valued exponentially localized breather solutions $U:\R^3\times \R\to \C^3$ of the form $U(x,t)=\phi(g(x))e^{\mathrm{i}\omega t}\frac{\nabla g(x)}{|\nabla g(x)|}$ with a suitable real-valued profile $\phi:\Rg(g)\to \R$. The proof is by an explicit construction which is virtually the same as in \cite[Theorem 3]{plum_reichel}.  
\label{remarks_breather}
\end{remark}

\medskip

For \eqref{rogue} we can show the existence of another type of exponentially localized breather having a possibly non-zero norm-limit at infinity. For such solutions we are not aware of any other existence result for \eqref{rogue}. We begin with the constant coefficient case.

\begin{theorem} Let $s\equiv 1$, $q\equiv 1$, and $V\equiv 1$. For $T\geq \frac{2\pi}{\sqrt{p-1}}$ there is a $T$-periodic weak solution $U^\ast(x,t) = y(t) \frac{\nabla g(x)}{|\nabla g(x)|}$ of \eqref{rogue} where $y$ is a positive $T$-periodic solution of $\ddot y -y+|y|^{p-1}y=0$.
\label{dark_constant_coefficients}
\end{theorem}

\begin{remark} For $T\in (0,\frac{2\pi}{\sqrt{p-1}})$ the exists also a $T$-periodic weak solution of \eqref{rogue} of the above type with sign-changing $y$, cf. Figure~\ref{fig:phase_plane_rogue}. 
\end{remark} 

In case of non-constant coefficients our result is the following.

\begin{theorem} Let $s,q,V:\R^3\to (0,\infty)$ and $g:\R^3\to\R$ satisfy {\rm (C1)--(C3)}, {\rm (B1), (B2)'} and suppose additionally that 
\begin{itemize}
\item[(B3)'] $\tau_\infty:= \lim_{|x|\to\infty}\tau(x)$ exists and and $\sup_{\R^3} |\tau(x)-\tau_\infty|e^{\delta|x|}<\infty$ for some $\delta>0$.
\end{itemize}
Then for $0<\omega\leq \sigma_\infty\sqrt{p-1}$ there exist a $T=\frac{2\pi}{\omega}$-periodic $\R^3$-valued weak solution $U$ of \eqref{rogue} (called breather) which is spatially exponentially localized w.r.t. $U_\infty(x,t)=\tau_\infty U^\ast(x,\sigma_\infty t)$ for some $\sigma_\infty T$-periodic solution $U^\ast$ from Theorem~\ref{dark_constant_coefficients}. The breather generates a continuum of phase-shifted breathers $U_a(x,t)=U(x,t+a(g(x)))$ where $a:\R\to \R$ is an arbitrary continuous function.
\label{main_breather_dark}
\end{theorem}

\begin{remark} \label{remark_main_breather_dark}
(a) Under slightly weaker assumptions one can directly construct monochromatic complex-valued breather solutions $U:\R^3\times \R\to \C^3$ which are of the form $U(x,t)=\phi(g(x))e^{\mathrm{i}\omega t}\frac{\nabla g(x)}{|\nabla g(x)|}$. In fact, one can drop (B2)' and take $\omega\geq 0$ arbitrary. Then the profile $\phi$ can be taken as $\phi(\zeta) := (\frac{\omega^2}{\tilde\sigma(\zeta)^2}+1)^\frac{1}{p-1}\tilde\tau(\zeta)$ with $\zeta\in\range(g)$. It implies that $U(x,t)$ is exponentially localized with respect to $U_\infty(x,t)= (\frac{\omega^2}{\sigma_\infty^2}+1)^\frac{1}{p-1}\tau_\infty e^{\mathrm{i}\omega t}\frac{\nabla g(x)}{|\nabla g(x)|}$.\\
(b) For $\omega=0$ the construction in (a) yields the stationary solution $U(x)=\tau(x)\frac{\nabla g(x)}{|\nabla g(x)|}$ which is exponentially localized w.r.t $U_\infty(x)=\tau_\infty\frac{\nabla g(x)}{|\nabla g(x)|}$. A second, time-periodic solution $U$, which is also exponentially localized w.r.t. the same function $U_\infty$ is given by Theorem~\ref{main_breather_dark} by taking $\omega=\sigma_\infty\sqrt{p-1}$. 
\end{remark}

In Section~\ref{sec:rogue} we consider solutions (called rogue waves), which are simultaneously localized in space and time. For rogue waves solutions of \eqref{rogue} our main result is the following, and we are not aware of any other existence result.

\begin{theorem} Let $s,q,V:\R^3\to (0,\infty)$ and $g:\R^3\to\R$ satisfy {\rm (C1)--(C3)}. Suppose moreover that
\begin{itemize}
\item[(R)] $\inf_{\R^3}\sigma>0$, $\sup_{\R^3}\tau(x) e^{\delta |x|}<\infty$ for some $\delta>0$.
\end{itemize}
Then there exists an $\R^3$-valued weak solution $U\not\equiv 0$ of \eqref{rogue} on $\R$ (called rogue wave) which is space-time exponentially localized w.r.t. $U_\infty=0$. The rouge wave generates a continuum of phase-shifted rogue waves $U_a(x,t)=U(x,t+a(g(x)))$ with the same space-time exponential localization provided $a:\R\to \R$ is an arbitrary continuous function with $\sup_{\R^3} \frac{|a(g(x))|}{1+|x|}<\infty$. Each of the rogue waves $U_a$ can be approximated (locally uniformly in $x$ and $t$) by a family of $T$-periodic solutions of \eqref{rogue} when $T\to \infty$. However, these $T$-periodic solutions are not localized in space when, e.g., $\sigma\in L^\infty(\R^3)$.
\label{main_rogue}
\end{theorem}

\begin{remark} (a) Recall from Remark~\ref{remarks_breather} the definition of the set $D$ of all finite accumulation points of $g$ as $|x|\to\infty$. Condition (R) shows that $\tilde\tau(\zeta)$ has to tend to $0$ as $\dist(\zeta,D)\to 0$. And since for the first example in Table~\ref{list} we have $D=\R$ there is no positive and continuous function $\tau=\tilde\tau\circ g$ with property (R) in this case.\\
(b) An explicit example of a family of rogue waves can be given as follows. Assume $s\equiv q\equiv 1$, $p=3$ and $V=V(|x|)\geq Ce^{\delta |x|}$ for some $C,\delta>0$. Then 
$$
U(x,t) = \frac{\sqrt{2}}{\sqrt{V(|x|)}\cosh(t)}\frac{x}{|x|}
$$
is a rogue wave solution of \eqref{rogue} which is space-time exponentially localized w.r.t. $U_\infty=0$. Clearly, $-U$ also solves \eqref{rogue}. 
\label{remarks_rogue}
\end{remark}

The paper is organized as follows. In Section~\ref{sec:breather} we consider breather solutions of \eqref{breather}$_\pm$ and \eqref{rogue}, and we prove the existence results of Theorem~\ref{main_breather}, Theorem~\ref{dark_constant_coefficients}, and Theorem~\ref{main_breather_dark}. Section~\ref{sec:rogue} deals with rogue waves of \eqref{rogue} and contains the proof of Theorem~\ref{main_rogue}. In the Appendix we give the proof of a technical result which is used in the proof of Theorem~\ref{main_breather_dark}.

\section{Proof of the main results for breathers} \label{sec:breather}

We begin by looking for solutions $U$ of \eqref{breather}$_\pm$ of the form $U(x,t) := \psi(g(x),t)\frac{\nabla g(x)}{|\nabla g(x)|}$ under the assumptions (C1)--(C3) as well as (B1) and (B2), (B2)', respectively, and (B3) of Theorem~\ref{main_breather}. In order to explain the underlying idea, we start with a formal calculation which will be made rigorous later. Due to the assumption $|\nabla g(x)|=G(g(x))$ a.e. in $\R^3$ we see that for fixed $t$
$$
U(x,t)= \nabla_x F_t(g(x)) \mbox{ where } F_t'(\zeta)= \psi(\zeta,t)/G(\zeta),
$$
i.e., $F_t(\zeta)$ is a primitive of $\psi(\zeta,t)/G(\zeta)$. With $U$ being a gradient field we find that $U$ solves \eqref{breather}$_\pm$ if and only if the function $\psi=\psi(\zeta,t):\Rg(g)\times\R\to\R$ satisfies 
\begin{equation}
\usetagform{pm}
\label{ode_psi}
\tilde s(\zeta) \ddot \psi(\zeta,t) + \tilde q(\zeta) \psi(\zeta,t) \pm \tilde V(\zeta)|\psi(\zeta,t)|^{p-1}\psi(\zeta,t) = 0 \mbox{ for } (\zeta,t)\in \Rg(g)\times\R
\end{equation}
where $\ddot\psi(\zeta,t)$ stands for $\frac{d^2}{dt^2}\psi(\zeta,t)$. The next obvious reduction is to use the fact that \eqref{ode_psi}$_\pm$ is autonomous and $\zeta=g(x)$ just acts as a parameter. Thus, rescaling \eqref{ode_psi}$_\pm$ suggests to set $\psi(\zeta,t):=\tilde \tau(\zeta)y(\tilde\sigma(\zeta)t)$ where $\tilde\sigma=\sqrt{\frac{\tilde q}{\tilde s}}$, $\tilde\tau=\bigl(\frac{\tilde q}{\tilde V}\bigr)^{\frac{1}{p-1}}$ and $y$ satisfies 
\begin{equation}
\usetagform{pm}
\ddot y + y \pm |y|^{p-1}y = 0.
\label{ode_y}
\end{equation}
The analysis of \eqref{ode_y}$_\pm$ is well--known and we recall next from our previous paper the most important facts, cf. Lemma 4 and Lemma 5 in \cite{plum_reichel}. For the reader's convenience we also present in Figure~\ref{fig:phase_plane_old} plots of the phase plane of \eqref{ode_y}$_\pm$. 

\begin{lemma}
Define the functions $A_\pm: \R^2\to \R$ by $A_\pm(\xi,\eta):= \eta^2+\xi^2\pm\frac{2}{p+1}|\xi|^{p+1}$. Then $A_\pm$ is a first integral for \eqref{ode_y}$_\pm$, i.e., every solution $y$ of \eqref{ode_y}$_\pm$ satisfies $A_\pm(y,\dot y)=\const=c$ for some $c\in\R$. Every bounded orbit of \eqref{ode_y}$_\pm$ is uniquely characterized by the value of $c$ in the range of $A_\pm$. Whenever a solution $y$ on such an orbit is periodic and non-stationary let the minimal period be $L_\pm(c)$ and the maximal amplitude $N_\pm(c):= \max_{t\in \R}|y(t)|$. Then, for the ``+''-case we have: 
\begin{itemize}
\item[(i)] $L_+:(0,\infty)\to (0,2\pi)$ is continuous and strictly decreasing with $\lim_{c\to \infty} L_+(c)=0$ and $L_+(0+) = 2\pi$. 
\item[(ii)]$N_+:[0,\infty)\to [0,\infty)$ is continuous and $N_+(c)\leq \sqrt{c}$ for all $c\geq 0$.
\item[(iii)] $M_+= L_+^{-1}: (0,2\pi)\to (0,\infty)$ has the expansion $\sqrt{M_+(s)} = \sqrt{\alpha}(2\pi-s)^\frac{1}{p-1}(1+O(2\pi-s))$ as $s \to 2\pi-$ for some constant $\alpha>0$. 
\end{itemize}
For the ``-'' case one finds:
\begin{itemize}
\item[(i)] $L_-:(0, \frac{p-1}{p+1})\to (2\pi,\infty)$ is continuous and strictly increasing with $\lim_{c\to \frac{p-1}{p+1}} L_-(c)=\infty$ and $L_-(0+) = 2\pi$.  
\item[(ii)] $N_-:[0, \frac{p-1}{p+1})\to [0,1)$ is continuous  and $N_-(c)\leq \sqrt{\frac{p+1}{p-1}c}$ for all $c\in [0,\frac{p-1}{p+1})$.
\item[(iii)] $M_-= L_-^{-1}: (2\pi,\infty) \to (0,\frac{p-1}{p+1})$ has the expansion $\sqrt{M_-(s)} = \sqrt{\alpha}(s-2\pi)^\frac{1}{p-1}(1+O(s-2\pi))$ as $s \to 2\pi+$ for the same constant $\alpha>0$ as in the ``+'' case.  
\end{itemize}
\label{phase_plane}
\end{lemma}

\begin{figure}[ht!]
\centering
\scalebox{1.0}{\includegraphics{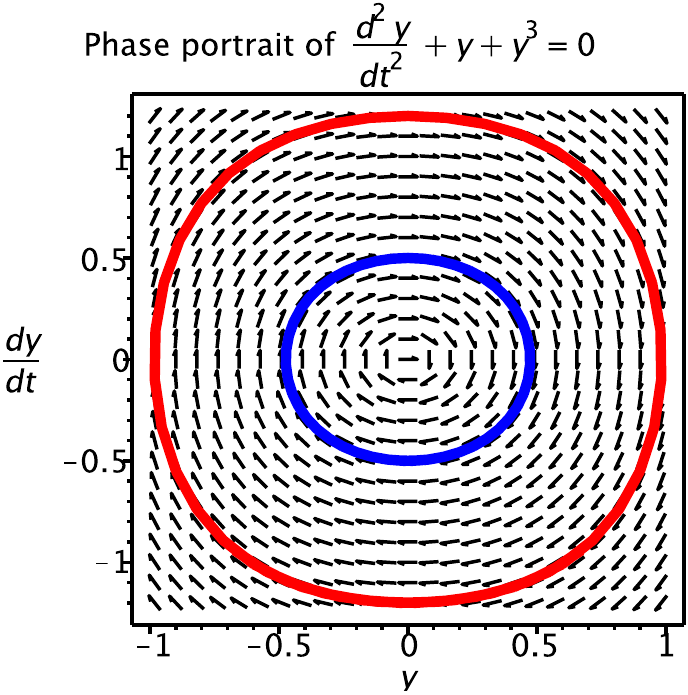}}\quad\scalebox{1.0}{\includegraphics{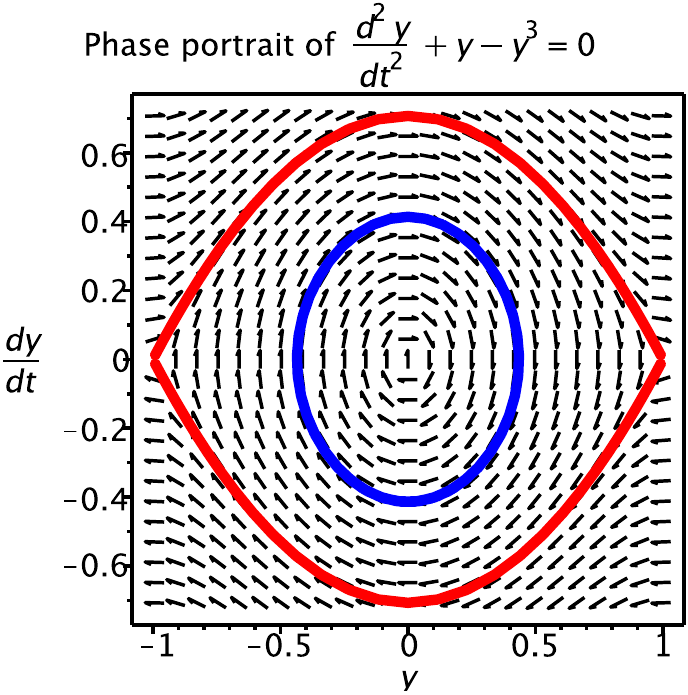}}
\caption{Part of the phase plane of \eqref{ode_y}$_\pm$ for $p=3$. Left: ``$+$''case with periodic orbits (blue, red). Right: ``$-$''case with periodic orbit (blue), two heteroclinic orbits (red).}
\label{fig:phase_plane_old}
\end{figure}

\medskip

\noindent
{\bf Proof of Theorem~\ref{main_breather}:} We give the proof only in the ``+'' case and indicate at the end of the proof the necessary changes for the ``-'' case. We begin by choosing a continuous curve $\gamma: [0,\infty)\to \R^2$ in phase space such that $A_+(\gamma(c))=c$, where $A_+$ is the first integral from Lemma~\ref{phase_plane}. Such a curve is e.g. given by $\gamma(c)=(0,\sqrt{c})$. There is a continuum of other possible choices of $\gamma$. The choice of $\gamma$ actually only selects a particular member of the continuum of phase-shifted breathers (we will comment on this aspect after the end of the proof in Remark~\ref{details}).

\medskip

Let us denote by $y(t;c)$ the solution of \eqref{ode_y}$_+$ with $\bigl(y(0;c),\dot y(0;c))\bigr)=\gamma(c)$. Then $y:\R\times [0,\infty)\to \R$ is a $C^2$-function and $y(t;c)$ is $L_+(c)$-periodic in the $t$-variable. Now we define the solution $\psi:\Rg(g)\times\R\to\R$ of \eqref{ode_psi}$_+$ by 
\begin{equation}
\psi(\zeta,t) := \tilde\tau(\zeta) y(\tilde\sigma(\zeta)t; c) \quad \mbox{ with } \quad \tilde\sigma(\zeta) = \left(\frac{\tilde q(\zeta)}{\tilde s(\zeta)}\right)^{1/2}, \quad \tilde\tau(\zeta) = \left(\frac{\tilde q(\zeta)}{\tilde V(\zeta)}\right)^\frac{1}{p-1}.
\label{def_psi_plus}
\end{equation}
The requirement of $T$-periodicity of $\psi$ in the $t$-variable tells us how to choose $c$ as a function of the variable $\zeta\in \Rg(g)$, i.e.,
$$
\tilde\sigma(\zeta) T \stackrel{!}{=} L_+(c). 
$$
Recall from Lemma~\ref{phase_plane} the definition $M_+=L_+^{-1}$ and that $M_+$ has a continuous extension $M_+:(0,2\pi]\to \R$ which is strictly decreasing. Now
\begin{equation}
\label{def_c_plus}
c(\zeta) := M_+(\tilde\sigma(\zeta)T) \quad \mbox{ with } T=\frac{2\pi}{\omega}
\end{equation}
has to be inserted into \eqref{def_psi_plus}. Note that the assumption (B2) of Theorem~\ref{main_breather} guarantees that $c(\zeta)$ is well-defined for $\zeta\in \Rg(g)$. Next we show that $\psi(g(x),t)$ is exponentially decaying to $0$ as $|x|\to \infty$. First note the estimate 
\begin{align*}
|\psi(g(x),t)| & \leq \tau(x) N_+(c(g(x)))  \\
 & \leq \tau(x) \sqrt{c(g(x))} \mbox{ by Lemma~\ref{phase_plane}(ii)}\\
 & \leq (\sup_{\R^3} \tau) \sqrt{M_+(\sigma(x)T)}. 
\end{align*}
By assumption (B1) of Theorem~\ref{main_breather} the argument of $M_+$ in the above inequality tends to $2\pi$ as $|x|\to \infty$. By Lemma~\ref{phase_plane}(iii) we have the estimate 
\begin{equation}
\label{est_psi_plus}
|\psi(g(x),t)| \leq  (\sup_{\R^3}\tau) \sqrt{\alpha}\left(2\pi-\sigma(x)T\right)^\frac{1}{p-1}O(1)  \mbox{ as } |x|\to \infty.
\end{equation}
Using again assumption (B1) from Theorem~\ref{main_breather} the above estimate yields 
$|\psi(g(x),t)| \leq  C \exp(-\tilde\delta |x|)$ for $x\in\R^3$ and some $C,\tilde\delta>0$ which proves the exponential decay of \begin{equation} \label{def_sol}
U(x,t):=\psi(g(x),t)\frac{\nabla g(x)}{|\nabla g(x)|}
\end{equation}
as $|x|\to\infty$. 

\medskip

Next we verify that $U$ as in \eqref{def_sol} is indeed a weak breather solution of \eqref{breather}$_+$. First note that (B1), continuity and positivity of $\sigma$ imply that $\inf_{\R^3} \sigma>0$ and the same also holds for $\tilde\sigma$. This implies that the function $x \mapsto c(g(x))$ with $c$ from \eqref{def_c_plus} is continuous and bounded on $\R^3$. Using that $\inf_{\Rg(g)} G>0$ by (C2) we obtain that 
$$
\zeta \mapsto f_t(\zeta):= \tilde\tau(\zeta) y(\tilde\sigma(\zeta)t;c(\zeta))\frac{1}{G(\zeta)}
$$
as well as
$$
\zeta \mapsto \psi(\zeta,t)= \tilde\tau(\zeta) y(\tilde\sigma(\zeta)t;c(\zeta))
$$
are (uniformly w.r.t. $t\in [0,T]$) bounded and continuous functions of $\zeta\in\Rg(g)$. If we denote by $F_t:\Rg(g)\to \R$ a primitive function of $f_t:\Rg(g)\to (0,\infty)$ then the chain rule for $W^{1,1}_{loc}$-functions, cf. \cite[Lemma 7.5]{gilbarg_trudinger}, tells us that $F_t\circ g \in W^{1,1}_{loc}(\R^3)$ and 
$$
\nabla (F_t\circ g) = (f_t\circ g)\nabla g = U(\cdot,t)
$$
because $F_t$ is continuously differentiable with bounded derivative and $g\in W^{1,1}_{loc}(\R^3)$ by definition. This implies that in a distributional sense $\nabla\times U=0$ as the following calculation for $\phi\in C_c^\infty(\R^3)$ shows:
\begin{align*}
\int_{\R^3} U(x,t)\cdot \nabla\times \phi(x)\,dx &= \int_{\R^3} \psi(g(x),t)\frac{\nabla g(x)}{|\nabla g(x)|} \cdot \nabla\times \phi(x)\,dx  \\
&= \int_{\R^3} \nabla (F_t\circ g)(x)\cdot \nabla\times \phi(x)\,dx \\
& = - \int_{\R^3} F_t(g(x)) \underbrace{\nabla\cdot\nabla\times}_{=0} \phi(x)\,dx = 0. 
\end{align*}
Hence we have found that $\nabla\times U=0$ for every fixed $t\in [0,T]$. Together with continuity, boundedness and exponential decay of the map $\R^3\times [0,T]\ni(x,t) \mapsto \psi(g(x),t)$, we see that $U\in C^2([0,T];\Hcurl(\R^3))$. The fact that $U$ is a weak solution of \eqref{breather} then reduces to multiplying \eqref{ode_psi}$_+$ with a compact support function $v\in\Hcurl(\R^3)$ and integrating over $\R^3$.  

\medskip

The asserted continuum of solutions $U(x,t+a(g(x)))$ arising from arbitrary continuous functions $a:\R\to\R$ is a direct consequence of the fact that \eqref{ode_psi}$_\pm$ is autonomous with respect to $t$ and that $\zeta=g(x)$ plays the role of a parameter. Additionally, it is important to note that now $U(x,t+a(g(x)))$ is the $x$-gradient of the function $\tilde F_t\circ g$, where $\frac{d}{d\zeta}\tilde F_t(\zeta)= \tilde\tau(\zeta)y(\tilde\sigma(\zeta)(t+a(\zeta));c(\zeta))\frac{1}{G(\zeta)}$. 

\medskip

Finally, let us comment on the changes that are necessary in the ``-''case. Here the parameter $c$ of the first integral ranges in $[0,\frac{p-1}{p+1})$ and the period function $L_-:[0,\frac{p-1}{p+1})\to [2\pi,\infty)$ is continuous and strictly increasing with inverse $M_-=L_-^{-1}: [2\pi,\infty)\to [0,\frac{p-1}{p+1})$. The choice of the curve $\gamma: [0,\frac{p-1}{p+1})\to \R^2$ in phase space is again such that $A_-(\gamma(c))=c$, but additionally we require $\gamma(0)=(0,0)$ in order to ensure that the curve $\gamma$ hits the small periodic orbits inside the two heteroclinics connecting $(\pm 1,0)$. In the ``+''case the normalization $\gamma(0)=(0,0)$ was automatically fulfilled. In \eqref{def_c_plus} we replace $M_+$ with $M_-$ and this time assumption (B2)' guarantees that $\zeta\mapsto c(\zeta)$ is well-defined on $\Rg(g)$. The definition of $\psi$ in \eqref{def_psi_plus} remains the same and the verification of its properties as well as the properties of the solution $U$ follows exactly the same lines as before. This finishes the proof of Theorem~\ref{main_breather}. \qed     

\begin{remark} \label{details}
(a) Here we give some details on the observation that the constructed breathers may have compact support. We assume in addition to the assumptions of Theorem~\ref{main_breather} that the coefficients $q,s$ are chosen in such way that $\supp(\tilde\sigma-\omega)\subset (-R,R)$ and that (for simplicity of the example) $|g(x)|\to \infty$ for $|x|\to \infty$ (weaker assumptions on $g$ are also possible). Note from Lemma~\ref{phase_plane} that $L_\pm(0)=2\pi$ and hence $M_\pm(2\pi)=0$. Therefore, whenever $\zeta\in \Rg(g)$ is such that $\tilde \sigma(\zeta)=\omega=\frac{2\pi}{T}$ then from \eqref{def_c_plus} we obtain $c(\zeta)=0$ and consequently $\psi(\zeta,\cdot)\equiv 0$. By choosing a suitable $\rho>0$ we find for $|x|\geq \rho$ that $|\zeta|=|g(x)|\geq R$ and thus $U(x,t+a(g(x)))=0$. Hence the entire family $U(x,t+a(g(x)))$ has compact support in $\overline{B}_\rho(0)$. \\
(b) As in \cite{plum_reichel} we will now comment on the choice of the initial curve $\gamma(c)=(0,\sqrt{c})$ which led to the solution family $y(t;c)$ such that $(y(0;c), \dot y(0;c))=\gamma(c)$. Our objective was to determine \emph{some} continuous curve such that $A_+(\gamma(c))=c$. The particular choice $\gamma(c)=(0,\sqrt{c})$ is convenient but arbitrary. Let us explain other possible choices of $\gamma$. E.g., using our previous choice for $y$, we may take
$$
\hat \gamma(c) := \bigl(y(b(c);c), \dot y(b(c);c)\bigr)
$$
for an arbitrary function $b\in C([0,\infty);\R)$. Clearly, $A_+(\hat\gamma(c))=A_+\left(y(b(c);c), \dot y(b(c);c)\right)=c$ since $A_+$ is a first integral of \eqref{ode_y}$_+$. With the new curve $\hat\gamma$ we can define a new solution family $\hat y(t;c)$ through the initial conditions
$$
\bigl(\hat y(0;c), \dot{\hat y}(0;c)\bigr) = \hat\gamma(c)
$$
By uniqueness of the initial value problem the new and old solution families have the simple relation
$$
\hat y(t;c) = y(t+b(c);c).
$$
In order to see the effect of the choice of the new curve let us compare the solutions $U$, $\hat U$ generated by $\gamma$, $\hat\gamma$, i.e.,
$$
U(x,t) = \tau(x) y(\sigma(x) t;c(g(x)))\frac{\nabla g(x)}{|\nabla g(x)|},
$$
where $c(\zeta)=L_+^{-1}(\tilde \sigma(\zeta)T)$. Likewise
\begin{align*}
\hat U(x,t) &= \tau(x) \hat y(\sigma(x)t; c(g(x)))\frac{\nabla g(x)}{|\nabla g(x)|} \\
&= \tau(x)y(\sigma(x)t+b(c(g(x)));c(g(x)))\frac{\nabla g(x)}{|\nabla g(x)|} \\
&= U(x,t+a(g(x))),
\end{align*}
where $a(\zeta) = b(c(\zeta))/\tilde\sigma(\zeta)$ is a continuous function on $\range g$. Hence, this different choice of the initial curve leads to a phase-shifted breather.
\end{remark}

\medskip

The proofs of Theorem~\ref{dark_constant_coefficients} and Theorem~\ref{main_breather_dark} follow the same idea and we look for breather solutions of \eqref{rogue} with the same ansatz $U(x,t) := \psi(g(x),t)\frac{\nabla g(x)}{|\nabla g(x)|}$ as in the proof of the previous theorem. Thus, in order to obtain that $U$ solves \eqref{rogue} the function $\psi=\psi(\zeta,t):\Rg(g)\times\R\to\R$ needs to satisfy  
\begin{equation}
\label{ode_psi_rogue}
\tilde s(\zeta) \ddot \psi(\zeta,t) - \tilde q(\zeta) \psi(\zeta,t) + \tilde V(\zeta)|\psi(\zeta,t)|^{p-1}\psi(\zeta,t) = 0 \mbox{ for } (\zeta,t)\in \Rg(g)\times\R.
\end{equation}
Rescaling \eqref{ode_psi_rogue} as before by setting $\psi(\zeta,t):=\tilde \tau(\zeta)y(\tilde\sigma(\zeta)t)$ with $\tilde\sigma, \tilde\tau$ as in \eqref{def_psi_plus} the function $y$ has to satisfy 
\begin{equation}
\ddot y - y + |y|^{p-1}y = 0.
\label{ode_y_rogue}
\end{equation}
The analysis of \eqref{ode_y_rogue} is not difficult due to the first integral
\begin{equation} \label{first_integral_rogue}
A(\xi,\eta) = \eta^2-\xi^2+\frac{2}{p+1}|\xi|^{p+1},
\end{equation}
i.e., $A(y,\dot y)=\const = c$ for every solution $y$ of \eqref{ode_y_rogue}. Orbits are uniquely characterized by the value of $c\in (\frac{1-p}{1+p},\infty)$. The origin is a saddle point attached to two homoclinic orbits corresponding to $c=0$. The points $(\pm 1,0)$ correspond to $c=\frac{1-p}{1+p}$ and are stable centers surrounded by ``small'' periodic orbits for $c\in (\frac{1-p}{1+p},0)$. The values $c>0$ corresponds to ``large'' periodic orbits surrounding the union of the two homoclinic orbits, cf. Figure~\ref{fig:phase_plane_rogue} for a sketch of the phase plane in the case $p=3$.   

\begin{figure}[ht!]
\centering
\scalebox{1.0}{\includegraphics{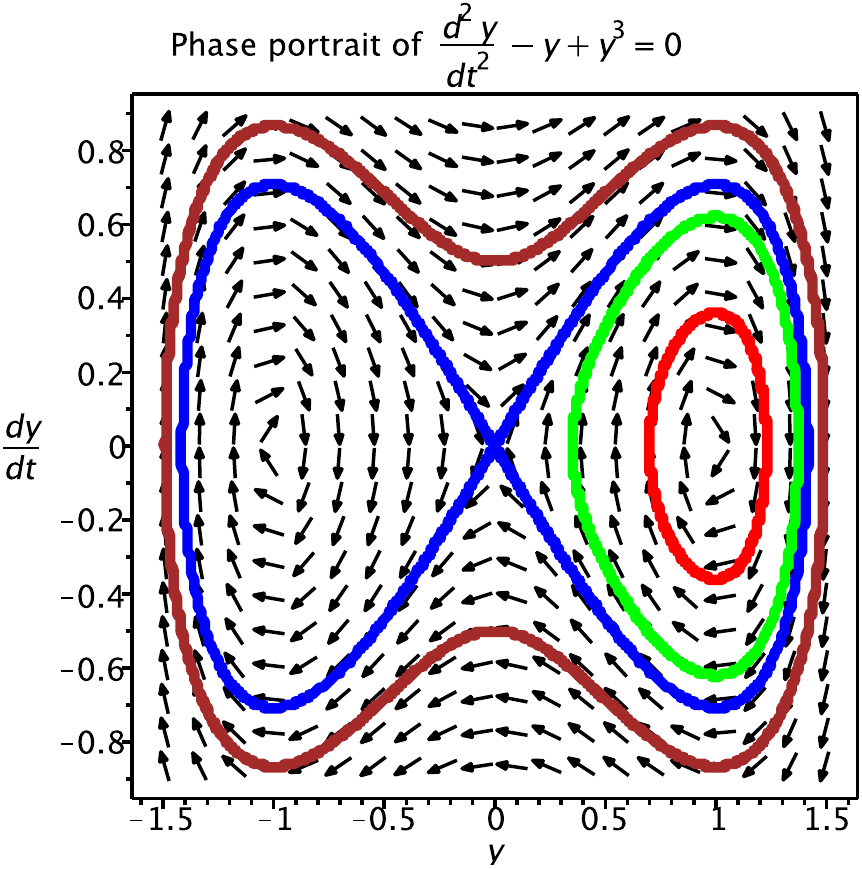}}
\caption{Part of the phase plane of \eqref{ode_y_rogue} for $p=3$ with ``small'' periodic orbits (green,red), ``large'' periodic orbit (brown) and homoclinic orbits (blue).}
\label{fig:phase_plane_rogue}
\end{figure}

The proofs of both Theorem~\ref{dark_constant_coefficients} and Theorem~\ref{main_breather_dark} rely on the following analogon of Lemma~\ref{phase_plane} for \eqref{ode_y_rogue}. Its proof is given in the Appendix.

\begin{lemma}
For every solution $y\geq 0$ on an orbit of \eqref{ode_y_rogue} given by $c\in (\frac{1-p}{1+p},0)$ let the minimal period be $L(c)$. Then we have:
\begin{itemize}
\item[(i)] $L:(\frac{1-p}{1+p},0)\to (\frac{2\pi}{\sqrt{p-1}},\infty)$ is continuously differentiable and strictly increasing with $L(\frac{1-p}{1+p}+) = \frac{2\pi}{\sqrt{p-1}}$,  $\lim_{c\to 0-} L(c)=+\infty$.
\item[(ii)] $M= L^{-1}: (\frac{2\pi}{\sqrt{p-1}},\infty)\to (\frac{1-p}{1+p},0)$ is continuously differentiable, strictly increasing, and $M(\frac{2\pi}{\sqrt{p-1}}+) = \frac{1-p}{1+p}$ and $M'(\frac{2\pi}{\sqrt{p-1}}+)= \frac{12(p-1)^{3/2}}{\pi p(p+3)}$.
\end{itemize}
\label{phase_plane_rogue}
\end{lemma}

Since \eqref{ode_y_rogue} is autonomous, we need to normalize the solutions $y=y(\cdot;c)$ by choosing for each value of $c$ initial conditions at $t=0$. This is somewhat arbitrary. The next lemma suggests a certain normalization and consequently develops continuity properties of the map $c\mapsto y(\cdot;c)$. 

\begin{lemma} \label{normalization}
Let $a: [1,(\frac{p+1}{2})^{1/(p-1)}] \to [\frac{1-p}{1+p},0]$ be the strictly increasing function given by $a(\xi):=-\xi^2+\frac{2}{p+1}\xi^{p+1}$ so that $A(a^{-1}(c),0)=c$ for all $c\in [\frac{1-p}{1+p},0]$. If we normalize the solutions $y(t;c)$ of \eqref{ode_y_rogue} by $y(0;c)=a^{-1}(c)$ and $\dot y(0;c)=0$ then the map $c\mapsto y(\cdot,c)$ is $\frac{1}{2}$-H\"older continuous uniformly for $c\in [\frac{1-p}{1+p},0]$ in the sense that 
$$
\|y(\cdot,c_1)-y(\cdot,c_2)\|_{L^\infty([0,T])} \leq C_T\sqrt{|c_1-c_2|}
$$
with a constant $C_T$ depending on the chosen time-interval $[0,T]$. 
\end{lemma}

\begin{proof} We first note that the function $c\mapsto (a^{-1}(c)-1)^2$ is continuously differentiable on the interval $[\frac{1-p}{1+p},0]$ since by a direct computation one can check that the derivative has limits at the endpoints of the interval. Since $c\mapsto (a^{-1}(c)-1)^2$ is Lipschitz one can easily check that $c\mapsto a^{-1}(c)$ is uniformly $\frac{1}{2}$-H\"older continuous on $[\frac{1-p}{1+p},0]$.

\medskip

Next we consider $c_1, c_2\in [\frac{1-p}{1+p},0]$ and set $z(t) := y(t;c_1)-y(t;c_2)$. Then $\ddot z -z+q(t) z=0$ with $q(t) = \frac{|y(t;c_1)|^{p-1} y(t;c_1) - |y(t;c_2)|^{p-1} y(t;c_2)|}{y(t;c_1)-y(t;c_2)}$. Since $y(t;c_1)$ and $y(t;c_2)$ are small periodic solutions inside the homoclinic they are both bounded from above by $(\frac{p+1}{2})^\frac{1}{p-1}$, and hence $|q(t)| \leq \frac{p(p+1)}{2}$. If we use that 
$$
\frac{d}{dt} (z^2+\dot z^2) = 2z\dot z +2\dot z (1-q(t))z = 2z\dot z (2-q(t)) \leq (z^2+\dot z^2)C 
$$
with $C=2+\frac{p(p+1)}{2}$ then a Gronwall argument implies $z(t)^2+ \dot z(t)^2 \leq (a^{-1}(c_1)-a^{-1}(c_2))^2 e^{TC}$. Therefore the uniform $\frac{1}{2}$-H\"older continuity of $c\mapsto a^{-1}(c)$ implies $|y(t;c_1)-y(t;c_2)|=|z(t)|\leq C_T\sqrt{|c_1-c_2|}$ for all $t\in [0,T]$ as claimed.
\end{proof}

\noindent
{\bf Proof of Theorem~\ref{dark_constant_coefficients}:} By Lemma~\ref{phase_plane_rogue} we know that for every given period $T> \frac{2\pi}{\sqrt{p-1}}$ there exists a value $c=M(T)=L^{-1}(T)$ such that on the level set $A^{-1}(c)$ there are two periodic orbits with period $T$, cf. Figure~\ref{fig:phase_plane_rogue}. This extends to $T=\frac{2\pi}{\sqrt{p-1}}$ where for $c=M(T)$ we have that $A^{-1}(c)$ consist of the two equilibria $(\pm 1,0)$. Choosing $y=y(t;c)$ to be the positive periodic orbit (positive equilibrium in case $T=\frac{2\pi}{\sqrt{p-1}}$) with $A(y,\dot y)=c$ the theorem follows if we set $U^\ast(x,t) := y(t;c)\frac{\nabla g(x)}{|\nabla g(x)|}$. 
\qed

\medskip

\noindent
{\bf Proof of Theorem~\ref{main_breather_dark}:}  First we choose a reference function $U_\infty$  at spatial infinity by $U_\infty(x,t)= \tau_\infty y(\sigma_\infty t; M(\sigma_\infty T))\frac{\nabla g(x)}{|\nabla g(x)|}$ where by assumption $\sigma_\infty T\geq \frac{2\pi}{\sqrt{p-1}}$. Here we use from Lemma~\ref{normalization} the normalized family $y(\cdot;c)$ of positive periodic orbits lying on the level set $A^{-1}(c)$ of the first integral function $A$. 

\medskip

Next we construct a $T=\frac{2\pi}{\omega}$-periodic solution of \eqref{ode_psi_rogue} by 
$$
\psi(\zeta,t):=\tilde \tau(\zeta)y(\tilde\sigma(\zeta)t;c(\zeta)) \quad \mbox{ with }\quad c(\zeta):=M(\tilde\sigma(\zeta)T).
$$
Note that $c(\zeta)$ is well defined for $\zeta\in\Rg(g)$ since by (B2)' we have $\tilde\sigma(\zeta)T\geq\sigma_\infty T\geq \frac{2\pi}{\sqrt{p-1}}$. In order to show that $U(x,t)=\psi(x,t)\frac{\nabla g(x)}{|\nabla g(x)|}$ is exponentially localized w.r.t. $U_\infty$ we need to estimate
\begin{align*}
|U(x,t)-U_\infty(x,t)| = & |\tau(x) y(\sigma(x)t;M(\sigma(x)T))-\tau_\infty y(\sigma_\infty t; M(\sigma_\infty T))| \\
\leq & |\tau(x)-\tau_\infty| y(\sigma(x)t;M(\sigma(x)T))  \\
& + \tau_\infty |y(\sigma(x)t;M(\sigma(x)T))-y(\sigma_\infty t; M(\sigma(x)T)| \\
& + \tau_\infty |y(\sigma_\infty t; M(\sigma(x)T)-y(\sigma_\infty t;M(\sigma_\infty T)|.
\end{align*}
Due to periodicity it is sufficient to consider the above estimate for $t\in [0,T]$. Now we use that $\|y(\cdot;c)\|_\infty \leq (\frac{p+1}{2})^\frac{1}{p-1}$, $\|\dot y(\cdot;c)\|_\infty \leq \sqrt{\frac{p-1}{p+1}}$, $\sup_{\R^3} \sigma=\Sigma<\infty$, and employ Lemma~\ref{normalization}, and (B1), (B3)' to estimate further 
\begin{align*}
|U(x,t)-U_\infty(x,t)| \leq & |\tau(x)-\tau_\infty| (\frac{p+1}{2})^\frac{1}{p-1} + \tau_\infty \|\dot y(\cdot, M(\sigma(x)T)\|_\infty (\sigma(x)-\sigma_\infty)T \\
& + \tau_\infty C_T\sqrt{|M(\sigma(x)T)-M(\sigma_\infty T)|} \\
\leq & C_1 |\tau(x)-\tau_\infty| + C_2 |\sigma(x)-\sigma_\infty| + \tau_\infty C_T\sqrt{T}\sqrt{\sigma(x)-\sigma_\infty}\sqrt{\sup_{[\frac{2\pi}{\sqrt{p-1}}, \Sigma T]} M'}  \\
\leq & C e^{-\frac{\delta}{2}|x|}.
\end{align*}
This finishes the proof of the exponential localization of $U(x,t)$ w.r.t. $U_\infty$. The proof that $U$ is a weak solution of \eqref{rogue} follows the line of the proof of Theorem~\ref{main_breather}.  
\qed

\section{Proof of the main result for rogue waves} \label{sec:rogue}

Now we work under the assumptions (C1)--(C3) and (R) of Theorem~\ref{main_rogue} and look for rogue wave solutions $U$ of \eqref{rogue} of the form $U(x,t) := \psi(g(x),t+a(g(x)))\frac{\nabla g(x)}{|\nabla g(x)|}$. As before we see that $U$ is a gradient field. 
\medskip

\noindent
{\bf Proof of Theorem~\ref{main_rogue}:} Let $y_0:\R\to (0,\infty)$ denote the positive homoclinic solution of \eqref{ode_y_rogue} with $y_0(0)=\bigl(\frac{p+1}{2}\bigr)^{\frac{1}{p-1}}$, $\dot y_0(0)=0$ so that $A(y_0(t),\dot y_0(t))=0$ (any other initial value except $(0,0)$ on the homoclinic orbit as well as the negative homoclinic orbit would also work). Asymptotically, we have $|y_0(t)|=e^{-|t|}(1+o(1))$ as $|t|\to \infty$ and hence $|y_0(t)|\leq C_0e^{-|t|}$ for all $t\in \R$ and some $C_0>0$. Then
\begin{equation} \label{def_psi_rogue}
\psi(\zeta,t) := \tilde\tau(\zeta) y_0(\tilde\sigma(\zeta)t)
\end{equation} 
solves \eqref{ode_psi_rogue} provided $\tilde\sigma, \tilde\tau$ are defined as in \eqref{def_psi_plus}. Therefore $U_0(x,t)=\psi(g(x),t)\frac{\nabla g(x)}{|\nabla g(x)|}$ provides a solution of \eqref{rogue}. Moreover, since $0\leq\tilde\tau(g(x))\leq Ce^{-\delta|x|}$ and $\inf_{\Rg(g)}\tilde\sigma=\tilde\sigma_\ast>0$ from assumption (R) we see that $U_0$ is space-time exponentially localized w.r.t. $U_\infty=0$. This establishes the existence of one particular rogue wave. Next we set
\begin{equation} \label{phase_shift_rogue}
U(x,t) := U_0(x,t+a(g(x)))=\psi(g(x),t+a(g(x)))\frac{\nabla g(x)}{|\nabla g(x)|}
\end{equation}
for an arbitrary continuous function $a:\Rg(g)\to \R$ with $|a(g(x))|\leq \tilde C(1+|x|)$ and some positive constant $\tilde C>0$. Repeating the details from the proof of Theorem~\ref{main_breather} one finds that $U$ is a weak solution of \eqref{rogue} on $\R^3\times\R$. Note that the fact that from a gradient-type solution $U_0(x,t)$ of \eqref{rogue} we can generate other gradient-type solutions $U$ of \eqref{rogue} by setting $U(x,t)=U_0(x,t+a(g(x))$ has already been exploited for \eqref{breather}$_\pm$ and it remains valid in the context of \eqref{rogue}. 

\medskip

Next we check that $U$ is also space-time exponentially localized. Using again assumption (R) and the bound $|y_0(t)|\leq C_0e^{-|t|}$ we find the estimate
\begin{align*}
|U(x,t)| & \leq \tilde\tau(g(x)) C_0 e^{-|\tilde\sigma(g(x))(t+a(g(x)))|} \\
& \leq  CC_0 e^{-\delta|x|-\tilde\sigma_\ast|t+a(g(x))|}.
\end{align*}
By the estimates
$$
\left\{\begin{array}{rll}
\tilde\sigma_\ast|t+a(g(x))| & \geq \frac{\tilde\sigma_\ast}{2}|t| & \mbox{ if } |t|\geq 2|a(g(x))|, \vspace{\jot}\\
\delta |x| & \geq \frac{\delta}{2}|x|+\frac{\delta}{4\tilde C}|t|-\frac{\delta}{2}& \mbox{ if } |t| \leq 2|a(g(x))|
\end{array}\right.
$$
we obtain
$$
|U(x,t)| \leq CC_0e^\frac{\delta}{2} e^{-\tilde\delta(|x|+|t|)}
$$
where $\tilde\delta=\min\bigl\{\frac{\delta}{2},\frac{\tilde\sigma_\ast}{2},\frac{\delta}{4\tilde C}\bigr\}$ which proves the claim.

\medskip

Finally, we explain that $U$ can be approximated by $T$-periodic solutions as $T\to\infty$. Recall from the properties of \eqref{ode_y_rogue} as explained in Lemma~\ref{phase_plane_rogue} that ``small'' periodic orbits of \eqref{ode_y_rogue} around the equilibrium $(1,0)$ inside the positive homoclinic are associated to negative values of the first integral $A(\xi,\eta)$ from \eqref{first_integral_rogue}. Recall also that the function $L:(\frac{1-p}{1+p},0)\to (\frac{2\pi}{\sqrt{p-1}},\infty)$, which assigns to each value $c$ of the first integral the minimal period of the orbit of \eqref{ode_y_rogue} with  $A(y,\dot y)=c$, strictly increases from $\frac{2\pi}{\sqrt{p-1}}$ at $c=\frac{1-p}{1+p}$ for the equilibrium $(1,0)$ to $+\infty$ at $c=0$ for the positive homoclinic, cf. Figure~\ref{fig:phase_plane_rogue} and Lemma~\ref{phase_plane_rogue}. By choosing\footnote{An example of such a function is given by $\gamma(c)=(a^{-1}(c),0)$ with $a(\xi)=-\xi^2+\frac{2}{p+1}\xi^{p+1}$ for $\xi\in [1,(\frac{p+1}{2})^{\frac{1}{p-1}}]$.} a continuous curve $\gamma:[\frac{1-p}{1+p},0]\to [0,\infty)\times\R$ of initial values with $A(\gamma(c))=c$ and such that $\gamma(0)=\bigl((\frac{p+1}{2})^{\frac{1}{p-1}},0\bigr)$ we can normalize the positive periodic orbits of \eqref{ode_y_rogue} by the requirement $(y(0;c),\dot y(0;c))=\gamma(c)$ for $c\in [\frac{1-p}{1+p},0]$. In particular $y(\cdot;0)=y_0(\cdot)$ where $y_0$ is the positive homoclinic from the beginning of the proof. Using that $M=L^{-1}$ from Lemma~\ref{phase_plane_rogue} has a continuous extension $M: [\frac{2\pi}{\sqrt{p-1}},\infty)\to [\frac{1-p}{1+p},0)$ we may define a $T$-periodic weak breather solution of \eqref{rogue} by 
$$
U_T(x,t) := \psi_T(g(x),t)\frac{\nabla g(x)}{|\nabla g(x)|} \mbox{ with } \psi_T(\zeta,t)=\tilde \tau(\zeta)y\bigl(\tilde\sigma(\zeta)t;M(\tilde\sigma(\zeta)T)\bigr)
$$
with $\tilde\sigma, \tilde\tau$ as in \eqref{def_psi_plus}. Note that the above construction requires $T>\frac{2\pi}{\sqrt{p-1}\tilde\sigma_\ast}$. Therefore, as $T\to \infty$ we get that $M(\sigma(x)T)\to 0$ uniformly in $x\in\R^3$ and thus $U_T(x,t)\to U_0(x,t)$ locally uniformly in $(x,t)\in\R^4$ as $T\to\infty$. Adding the function $a(g(x))$ in the time-variable as in \eqref{phase_shift_rogue} both to $U_T$ and $U_0$ yields the approximation claim. Note that if $\sigma\in L^\infty(\R^3)$ then for each admissible finite $T>0$ the breather solution $U_T$ is not localized in space since at every space point $x\in\R^3$ it oscillates according to a periodic orbit of \eqref{ode_y_rogue} which has a positive distance from the homoclinic as well as from the equilibrium $(1,0)$. \qed

\begin{remark} We could have given other approximations $U_T$ of $U_0$ by utilizing the ``large'' periodic orbits outside the two homoclinics, cf. Figure~\ref{fig:phase_plane_rogue}. However, it is not clear to us if and in which sense these periodic breathers converge to a rogue wave. 
\end{remark}

\section*{Appendix}

\noindent
{\bf Proof of Lemma~\ref{phase_plane_rogue}.} Let $k(y):= 1-y^2+\frac{2}{p+1}(y^{p+1}-1)$. Then $k'(y)=2(y^p-y)$ and $k''(y)=2(py^{p-1}-1)$. If we set $k_- := k|_{[0,1]}$ and $k_+ := k|_{[1,(\frac{p+1}{2})^{1/(p-1)}]}$ then $k_-$ is strictly decreasing from $\frac{p-1}{p+1}$ to $0$ on the interval $[0,1]$ and $k_+$ is strictly increasing from $0$ to $\frac{p-1}{p+1}$ on the interval $[1,(\frac{p+1}{2})^{1/(p-1)}]$. For $c\in [\frac{1-p}{p+1},0]$ we can therefore define $N_\pm(c)=k_\pm^{-1}(\tilde c)$ with the shorthand $\tilde c=c+\frac{p-1}{p+1}$. Then $N_\pm(c)$ with $N_-(c)\leq 1 \leq N_+(c)$ denote the two extreme points with speed $\dot y=0$ on the positive orbit whose first integral has the value $c$. Recall that $L(c)$ is the minimal time-period of an orbit parameterized by $c\in (\frac{1-p}{1+p},0)$. Since such an orbit has the symmetry that for $\xi>0, \eta\in\R$ we have that $A(\xi,\eta)=c$ if and only if $A(\xi,-\eta)=c$, we find the following expression 
\begin{equation}
\label{beginning}
L(c) = \int_0^{L(c)} dt = 2 \int_{N_-(c)}^{N_+(c)} \frac{dy}{\sqrt{c+y^2-\frac{2}{p+1}y^{p+1}}} = 2\int_{N_-(c)}^{N_+(c)} \frac{dy}{\sqrt{\tilde c-k(y)}}
\end{equation}
where we have used the substitution $\dot y(t)\,dt = dy$ and $A(y,\dot y)\equiv c$. 

\noindent
\emph{Step 1 -- differentiability of $L(c)$ and expression for $L'(c)$:} We claim that for $c\in (\frac{1-p}{1+p},0)$ the function $c\mapsto L(c)$ is continuously differentiable and that 
\begin{equation} \label{L_prime}
L'(c) = \frac{1}{\tilde c} \int_{N_-(c)}^{N_+(c)} \frac{k'(y)^2-2k(y)k''(y)}{k'(y)^2} \frac{dy}{\sqrt{\tilde c-k(y)}}.
\end{equation}
The proof is done by splitting the integral $L(c)=\int_{N_-(c)}^{N+(c)}\ldots\,dy= \int_{N_-(c)}^1\ldots \,dy+\int_1^{N_+(c)} \ldots \,dy$ and substituting $y=k_\mp^{-1}(\tilde c z)$ in the two integrals, respectively. This results in 
\begin{equation} \label{split_the_integral}
L(c) = -2\int_0^1 \frac{\sqrt{\tilde c} (k_-^{-1})'(\tilde c z)}{\sqrt{1-z}}\,dz + 2\int_0^1 \frac{\sqrt{\tilde c} (k_+^{-1})'(\tilde c z)}{\sqrt{1-z}}\,dz 
\end{equation}
Now we can consider differentiation w.r.t. $c$. We only show the result for the first of the two integrals. Using the formulas $(k_-^{-1})'= \frac{1}{k'(k_-^{-1})}$ and $(k_-^{-1})''=-\frac{k''(k_-^{-1})}{k'(k_-^{-1})^3}$ 
we find 
\begin{equation} \label{cara}
\begin{split} 
\frac{\partial}{\partial\tilde c} \frac{2\sqrt{\tilde c} (k_-^{-1})'(\tilde c z)}{\sqrt{1-z}} =& 
\frac{1}{\sqrt{\tilde c}} \left( (k_-^{-1})'(\tilde c z) + 2 \tilde c z (k_-^{-1})''(\tilde c z)\right)\frac{1}{\sqrt{1-z}}\\
=& \frac{1}{\sqrt{\tilde c}} \left(\frac{(k')^2-2kk''}{(k')^3}\right)(k_-^{-1}(\tilde cz))\frac{1}{\sqrt{1-z}}. 
\end{split}
\end{equation}
Next we see that $\left((k')^2-2kk''\right)'(y) = -2k(y)k'''(y) = -4p(p-1)y^{p-2} k(y)$ so that using $k(1)=0=k'(1)$ this implies 
\begin{equation} \label{first_note}
\left((k')^2-2kk''\right)(y)=-4p(p-1) \int_1^y t^{p-2}k(t)\,dt. 
\end{equation}
By a straightforward computation we see that $\frac{1}{k'(y)^3} \left(\int_1^y t^{p-2}k(t)\,dt\right)$ and hence $\frac{((k')^2-2kk'')(y)}{k'(y)^3}$ have limits as $y\to 1$. Therefore they are continuous functions on the entire interval $[N_-(c),N_+(c)]$. In particular, for $\tilde c$ in compact subintervals $J\subset(\frac{1-p}{1+p},0)$ there exists $C_J>0$ such that 
$$
\left|\frac{\partial}{\partial\tilde c} \frac{2\sqrt{\tilde c} (k_-^{-1})'(\tilde c z)}{\sqrt{1-z}}\right|\leq \frac{C_J}{\sqrt{1-z}} \quad \mbox{ for } z\in (0,1).
$$
This upper bound shows continuous differentiability by a dominated convergence argument and that for $c\in (\frac{1-p}{1+p},0)$
$$
\frac{d}{dc} \left(-2\int_0^1 \frac{\sqrt{\tilde c} (k_-^{-1})'(\tilde c z)}{\sqrt{1-z}}\,dz\right) = 
\frac{-1}{\sqrt{\tilde c}}\int_0^1 \left( (k_-^{-1})'(\tilde c z) + 2 \tilde c z (k_-^{-1})''(\tilde c z)\right)\,\frac{dz}{\sqrt{1-z}}.
$$
By \eqref{cara} and by reverting the substitution $z=\frac{1}{\tilde c}k(y)$ we get 
$$
\frac{d}{dc} \left(-2\int_0^1 \frac{\sqrt{\tilde c} (k_-^{-1})'(\tilde c z)}{\sqrt{1-z}}\,dz\right) = \frac{1}{\tilde c} \int_{N_-(c)}^1 \frac{k'(y)^2 - 2 k(y)k''(y)}{k'(y)^2}\, \frac{dy}{\sqrt{\tilde c-k(y)}}.
$$
Together with an analogous computation for the second integral in \eqref{split_the_integral} we obtain the claim of Step 1.

\medskip

\noindent
\emph{Step 2 -- alternative expression for $L'(c)$:} Next we claim that 
\begin{equation} \label{L_prime_alt}
\tilde c L'(c) = 8p(p-1) \int_{N_-(c)}^{N_+(c)} \frac{y^{p-2}}{k'(y)^4} \Phi(y) \sqrt{\tilde c-k(y)}\,dy
\end{equation}
where $\Phi(y):= \left[3y^{2-p}k''(y)\left(\int_1^y t^{p-2}k(t)\,dt\right)-k(y)k'(y)\right]$. Eventually, the formula will follow from an integration by parts. If we insert \eqref{first_note} in \eqref{L_prime} we obtain
\begin{align*}
\tilde c L'(c) = & -4p(p-1)\int_{N_-(c)}^{N_+(c)} \frac{1}{k'(y)^2} \left(\int_1^y t^{p-2}k(t)\,dt\right) \frac{dy}{\sqrt{\tilde c- k(y)}} \\
= & 8p(p-1) \int_{N_-(c)}^{N_+(c)} \frac{1}{k'(y)^3} \left(\int_1^y t^{p-2}k(t)\,dt\right) \frac{d}{dy}\sqrt{\tilde c- k(y)}\,dy\\
= & 8p(p-1) \frac{1}{k'(y)^3} \left(\int_1^y t^{p-2}k(t)\,dt\right) \sqrt{\tilde c- k(y)}\Big|_{N_-(c)}^{N_+(c)} \\
& - 8p(p-1) \int_{N_-(c)}^{N_+(c)} \frac{d}{dy} \left[\frac{1}{k'(y)^3}\left(\int_1^y t^{p-2} k(t)\,dt\right) \right] \sqrt{\tilde c- k(y)}\,dy.
\end{align*}
Recall that \eqref{first_note} implies that $\frac{1}{k'(y)^3} \left(\int_1^y t^{p-2}k(t)\,dt\right)$ is convergent at $y=1$ and hence continuous and bounded on the entire interval $[N_-(c),N_+(c)]$. Together with the fact that $\tilde c-k(y)$ vanishes for $y=N_\pm(c)$ this allows us to further compute  
$$
\tilde c L'(c) = 8p(p-1) \int_{N_-(c)}^{N_+(c)} \frac{y^{p-2}}{k'(y)^4}\left[3y^{2-p}k''(y)\left(\int_1^y t^{p-2}k(t)\,dt\right)-k(y)k'(y)\right] \sqrt{\tilde c-k(y)}\,dy.
$$
This proves the claim of Step 2. Note that the integration by parts is justified since in the last formula the function $\Phi(y)$ is $O(y-1)^4$ near $y=1$.

\medskip

\noindent
\emph{Step 3 -- monotonicity of $L(c)$ for $p\geq 2$:} We check the sign of $\Phi(y)$ in \eqref{L_prime_alt} by first computing its derivative
\begin{align*}
\Phi'(y) =& -3(p-2)y^{1-p} k''(y) \left(\int_1^y t^{p-2}k(t)\,dt\right) + 3y^{2-p} k'''(y)\left(\int_1^y t^{p-2}k(t)\,dt\right)\\
& + \underbrace{3k''(y)k(y)-k'(y)^2-k(y)k''(y)}_{=(2kk''-(k')^2)(y)}.
\end{align*}
Using \eqref{first_note} and the explicit form of $k$ and its derivatives we find 
\begin{equation} \label{final_formula}
\Phi'(y) = \underbrace{\left(6(p-2)y^{1-p}+2p(2p+1)\right)}_{>0 \text{ for } p\geq 2}\left(\int_1^y t^{p-2} k(t)\,dt\right).
\end{equation}
Hence $\Phi'(y)>0$ on $(1,\infty)$ and $\Phi'(y)<0$ on $(0,1)$. Since $\Phi(1)=0$ we get the desired result $\Phi>0$ on $(0,\infty)\setminus\{1\}$. 

\medskip

\noindent
\emph{Step 4 -- monotonicity of $L(c)$ for $1<p<2$:} Here we deduce the sign of $\Phi(y)$ in a slightly different way. Using the final formula \eqref{final_formula} from Step 3 we compute 
\begin{align*}
(y^{p-2}\Phi(y))' = & y^{p-2}\Phi'(y) +(p-2)y^{p-3}\Phi(y) \\
=& y^{p-2}\left(6(p-2)y^{1-p}+2p(2p+1)\right)\left(\int_1^y t^{p-2} k(t)\,dt\right) \\
& + (p-2)y^{p-3}\left[3y^{2-p}k''(y)\left(\int_1^y t^{p-2}k(t)\,dt\right)-k(y)k'(y)\right] \\
=& y^{p-2} 10p(p-1) \left(\int_1^y t^{p-2} y(t)\,dt\right) + (2-p)y^{p-3} k(y)k'(y).
\end{align*}
Since $1<p<2$ we see that both summands are negative on $(0,1)$ and positive on $(1,\infty)$. Using $\Phi(1)=0$ we deduce $y^{p-2}\Phi(y)>0$ on $(0,\infty)\setminus\{1\}$ as in Step 3.

\medskip

\noindent
\emph{Step 5 -- limit behaviour of $L(c)$:} Recall from \eqref{beginning} that
$$
L(c) = \int_0^{L(c)} dt = 2 \int_{N_-(c)}^{N_+(c)} \frac{dy}{\sqrt{c+y^2-\frac{2}{p+1}y^{p+1}}}.
$$
Using that $N_-(c)\to 0$ and $N_+(c)\to \left(\frac{p+1}{2}\right)^\frac{1}{p-1}$ as $c\to 0-$ together with Fatou's Lemma leads to  
\begin{equation} \label{period}
\liminf_{c\to 0-} L(c) \geq 2 \int_0^{(\frac{p+1}{2})^\frac{1}{p-1}} \frac{dy}{y\sqrt{1-\frac{2}{p+1}y^{p-1}}}=\infty.
\end{equation}
In order to complete the proof of Lemma~\ref{phase_plane_rogue}(i) it remains to expand $L(c)$ as $c\to \frac{1-p}{1+p}$. The double zero of $y\mapsto k(y)$ at $y=1$ shows that $k(y)=(p-1)(y-1)^2+ O((y-1)^3)$ as $y\to 1$ and hence 
\begin{equation} \label{hilft1}
\frac{\sqrt{k(y)}}{k'(y)} \to \pm \frac{1}{2\sqrt{p-1}} \mbox{ as } y \to 1\pm
\end{equation}
Applying l'Hospital's rule three times we also find 
\begin{equation} \label{hilft2}
\frac{\Phi(y)}{k'(y)^4} \to \frac{p+3}{48(p-1)^2} \mbox{ as } y\to 1.
\end{equation}
According to \eqref{split_the_integral}, using $(k_\pm^{-1})'=\frac{1}{k'_\pm(k_\pm^{-1})}$  and \eqref{hilft1} we get
\begin{align*}
L(c) & =  -2 \int_0^1 \frac{\sqrt{\tilde c z}}{k'(k_-^{-1}(\tilde cz))} \frac{dz}{\sqrt{z(1-z)}}+2 \int_0^1 \frac{\sqrt{\tilde c z}}{k'(k_+^{-1}(\tilde cz))} \frac{dz}{\sqrt{z(1-z)}} \\
 & =  -2 \int_0^1 \frac{\sqrt{k}}{k'}(k_-^{-1}(\tilde cz)) \frac{dz}{\sqrt{z(1-z)}}+2 \int_0^1 \frac{\sqrt{k}}{k'}(k_+^{-1}(\tilde cz)) \frac{dz}{\sqrt{z(1-z)}} \\
& \stackrel{c\to \frac{1-p}{1+p}}{\longrightarrow}  \frac{2}{\sqrt{p-1}} \int_0^1 \frac{dz}{\sqrt{z(1-z)}} = \frac{2\pi}{\sqrt{p-1}}.  
\end{align*}
Next we split the integral in \eqref{L_prime_alt} into $\int_{N_-(c)}^1\ldots+\int_1^{N_+(c)}\ldots$ and substitute $y=k_\pm^{-1}(\tilde c z)$, respectively. This implies 
\begin{align*}
L'(c)  =& -8p(p-1) \int_0^1 \left[y^{p-2} \frac{\Phi(y)}{k'(y)^4}\right]_{y=k_-^{-1}(\tilde cz)} \frac{\sqrt{\tilde cz}}{k'(k_-^{-1}(\tilde cz))} \sqrt{\frac{1-z}{z}}\,dz \\
& +8p(p-1) \int_0^1 \left[y^{p-2} \frac{\Phi(y)}{k'(y)^4}\right]_{y=k_+^{-1}(\tilde cz)} \frac{\sqrt{\tilde cz}}{k'(k_+^{-1}(\tilde cz))} \sqrt{\frac{1-z}{z}}\,dz.
\end{align*}
If we note that $\frac{\sqrt{\tilde cz}}{k'(k_\pm^{-1}(\tilde cz))}=\frac{\sqrt{k}}{k'}(k_\pm(\tilde cz))$ and use the relations \eqref{hilft1}, \eqref{hilft2} then we obtain 
$$
L'(c) \stackrel{c\to \frac{1-p}{1+p}}{\longrightarrow} \frac{p(p+3)}{6(p-1)^\frac{3}{2}} \int_0^1 \sqrt{\frac{1-z}{z}}\,dz = \frac{\pi p(p+3)}{12(p-1)^\frac{3}{2}}
$$
as claimed. 

\medskip

\noindent
\emph{Step 6 -- limit behaviour of $M(c)$ for $c\to \frac{2\pi}{\sqrt{p-1}}+$:} Step 5 shows that  $L(\frac{1-p}{1+p}+)=\frac{2\pi}{\sqrt{p-1}}$ and  $L'(\frac{1-p}{1+p}+)=\beta:= \frac{\pi p(p+3)}{12(p-1)^{3/2}}$. Therefore the inverse $M=L^{-1}$ has the property $M(\frac{2\pi}{\sqrt{p-1}}+)=\frac{1-p}{1+p}$ and $M'(\frac{2\pi}{\sqrt{p-1}}+)=1/\beta$. This completes the proof of Lemma~\ref{phase_plane_rogue}. \qed

\section*{Acknowledgments} Funded by the Deutsche Forschungsgemeinschaft (DFG, German Research Foundation) – Project-ID 258734477 – SFB 1173.
\bibliographystyle{plain}	
\bibliography{bibliography}

\end{document}